\documentclass{article}
\usepackage{amscd}
\usepackage{amsmath}
\usepackage{amssymb}
\usepackage{amsthm}
\usepackage{ascmac}
\usepackage{bm}
\usepackage{graphicx}
\usepackage{here}
\usepackage{mathtools}\numberwithin{equation}{section}
\usepackage{minitoc}
\usepackage{multicol}
\usepackage{multirow}
\usepackage{setspace}
\theoremstyle{definition}
\newtheorem{df}{Definition}

\newtheorem{eg}{Example}
\newtheorem{rmk}{Remark}
\theoremstyle{theorem}

\newtheorem{thm}{Theorem}

\newtheorem{lem}{Lemma}

\title{Several properties of summatory Ehrhart polynomials and series\\
of convex lattice polytopes}
\author{Takashi HIROTSU}
\date{\today}
\begin{document}
\maketitle
\begin{abstract}
In this article, for a convex lattice polytope, we further investigate the summatory function of its Ehrhart polynomial, which is called the summatory Ehrhart polynomial, and introduce its summatory Ehrhart series. 
We prove several fundamental properties of these invariants. 
In particular, we derive a summatory analogue of the classical Ehrhart--Macdonald reciprocity law, which establishes a signed functional equation between the polytope and its relative interior via the substitution $t \mapsto -t-1.$
\end{abstract}
\section{Introduction}
Let $n$ and $d$ be integers such that $0 \leq d \leq n.$ 
For each bounded convex set $S \subset \mathbb R^n$ and integer $t \geq 0,$ the {\itshape lattice-point enumerator} $\mathcal E(S;t)$ for the $t$-th dilate $tS$ of $S$ is defined as 
\[\mathcal E(S;t) := \# (tS\cap \mathbb Z^n).\] 
Throughout this article, let $P \subset \mathbb R^n$ be a convex lattice polytope with $\dim P = d$ and $\mathrm{codeg}\,P = c,$ where 
\[\mathrm{codeg}\,P := \min\{ t \in \mathbb Z_{> 0} \mid tP^\circ\cap\mathbb Z^n \neq \varnothing\}.\] 
\begin{thm}[{\cite[Theorem 3.8 and Corollaries 3.15, 3.17, and 3.20]{BR15}}]\label{thm-ehr-poly}
The lattice-point enumerator $\mathcal E(P;t)$ of $P$ is a polynomial in $t$ of the form 
\[\mathcal E(P;t) = c_dt^d+c_{d-1}t^{d-1}+\cdots +c_1t+c_0,\] 
where $d!c_i \in \mathbb Z$ for any $i \in \{ 0,1,\dots,d\},$ $c_d = \mathrm{vol}\,P,$ and $c_0 = 1.$
\end{thm}
The polynomial $\mathcal E(P;t)$ in Theorem \ref{thm-ehr-poly} is referred to as the {\itshape Ehrhart polynomial} of $P.$ 
Let $P^\circ$ denote the relative interior of $P.$ 
The {\itshape Ehrhart series} $\mathrm{Ehr}(P;z)$ of $P$ and $\mathrm{Ehr}(P^\circ;z)$ of $P^\circ$ are defined by 
\[\mathrm{Ehr}(P;z) := \sum_{t = 0}^{\infty}\mathcal E(P;t)z^t \quad\text{and}\quad \mathrm{Ehr}(P^\circ;z) := \sum_{t = 1}^{\infty}\mathcal E(P^\circ;t)z^t\] 
Furthermore, the {\itshape $h^*$-polynomial} $h^*(P;z)$ of $P$ is defined by 
\[ h^*(P;z) := (1-z)^{d+1}\mathrm{Ehr}(P;z).\] 
The codegree $c$ is related to this polynomial via 
\[ c = d+1-\deg h^*(P;z).\] 
We recall the following definition and theorem.
\begin{df}[{\cite[Definition 2]{Hir26}}]
Let $P \subset \mathbb R^n$ be a convex lattice polytope, and let $t \geq 0$ be an integer.
\begin{enumerate}
\item[(1)]
We call a polytope $G \subset \mathbb R^n$ a {\itshape giganture} of $P$ in $tP$ if $G$ is contained in $tP$ and similar to $P,$ allowing a scale factor of $0.$ 
In particular, when the scale factor is $0,$ the giganture degenerates into a single point, which is referred to as a {\itshape degenerate} giganture.
\item[(2)]
A giganture $G$ of $P$ is said to be {\itshape horizontal} if $P$ is transformed into $G$ by translation and nonnegative integral rescaling (or equivalently, if $G = iP+a$ for some integer $i \geq 0$ and some vector $a \in \mathbb Z^n$).
\item[(3)]
Let $\mathcal H(P;t)$ denote the number of horizontal gigantures of $P$ in $tP.$
\end{enumerate}
\end{df}
Let $V \subset \mathbb Z^n$ be the set of all vertices of $P,$ and let $\mathrm{Pyr}\,P$ denote the convex hull of $\{ (v,0) \mid v \in V\}\cup\{ (0,\dots,0,1)\}$ in $\mathbb R^{n+1}.$ 
\begin{thm}[{\cite[Theorem 6]{Hir26}}]\label{thm-sep-poly}
The function $\mathcal H(P;t)$ is a polynomial in $t$ of the form 
\begin{align} 
\mathcal H(P;t) &= \sum_{i = 0}^{t}\mathcal E(P;i) \label{eq-sep-epsum} \\ 
&= \mathcal E(\mathrm{Pyr}\,P;t) \label{eq-sep-eppyr} \\ 
&= a_{d+1}t^{d+1}+a_{d}t^{d}+\cdots +a_{1}t+a_{0}, \notag 
\end{align} 
where $(d+1)!a_i \in \mathbb Z$ for any $i \in \{ 0,1,\dots,d+1\},$ $a_{d+1} = \mathrm{vol}\,P/(d+1),$ and $a_{0} = 1.$
\end{thm}
\begin{rmk}
Let $\mathcal H^+(P;t)$ denote the number of nondegenerate horizontal gigantures of $P$ in $tP$ for each integer $t > 0.$ 
Then, Theorem \ref{thm-sep-poly} yields that $\mathcal H^+(P;t)$ is also a polynomial in $t$ and given by 
\[\mathcal H^+(P;t) = \mathcal H(P;t)-\mathcal E(P;t) = \sum_{i = 0}^{t-1}\mathcal E(P;i) = \mathcal E(\mathrm{Pyr}\,P;t-1).\] 
Furthermore, the leading coefficient of $\mathcal H^+(P;t)$ is also $\mathrm{vol}\,P/(d+1)$ and its constant term vanishes (see \cite[Remark 4]{Hir26}).
\end{rmk}
Based on Theorem \ref{thm-sep-poly}, we refer to $\mathcal H(P;t)$ as the {\itshape summatory Ehrhart polynomial} of $P.$
\begin{df}
Let $\mathcal H(P^\circ;t)$ denote the number of horizontal gigantures of $P$ in $tP^\circ$ for each integer $t > 0.$
\end{df}
\begin{df}
We define the {\itshape summatory Ehrhart series} $\mathrm{SE}(P;z)$ of $P$ and $\mathrm{SE}(P^\circ;z)$ of $P^\circ$ by 
\[\mathrm{SE}(P;z) := \sum_{t = 0}^{\infty}\mathcal H(P;t)z^t \quad\text{and}\quad \mathrm{SE}(P^\circ;z) := \sum_{t = 1}^{\infty}\mathcal H(P^\circ;t)z^t.\] 
Furthermore, let 
\[ h^{\#}(P;z) := (1-z)^{d+2}\mathrm{SE}(P;z).\] 
\end{df}
In the author's previous work (see \cite{Hir26}), the $d$-dimensional volume $\mu_d^{\mathrm{nl}}(P)$ of a normal-sized miniature of $P$ was explicitly determined as $\mu_d^{\mathrm{nl}}(P) = \binom{2d+1}{d}^{-1}\mathrm{vol}_d(P)$ on the basis of Theorem \ref{thm-sep-poly}. 
Later, this formula was generalized to its $r$-dimensional volume $\mu_r^{\mathrm{nl}}(P)$ as $\mu_r^{\mathrm{nl}}(P) = \binom{d+r+1}{r}^{-1}\mathrm{vol}_r(P)$ for $r \in \{ 0,1,\dots,d\}.$ 
In this article, we establish several properties of the summatory Ehrhart polynomial and series. 
The main results are as follows. 
In particular, we obtain the following reciprocity law by viewing $\mathcal H(P;t)$ as a polynomial.
\begin{thm}\label{thm-sep-recip}
We have 
\begin{align} 
\mathcal H(P^\circ;t) &= \mathcal E((\mathrm{Pyr}\,P)^\circ;t+1) \label{eq-sepi-epi} \\ 
&= (-1)^{d+1}\mathcal H(P;-t-1). \label{eq-sepi-sep} 
\end{align} 
\end{thm}
\begin{thm}\label{thm-ses-ser}
We have 
\begin{align} 
\mathrm{SE}(P;z) &= \mathrm{Ehr}(\mathrm{Pyr}\,P;z) \notag \\ 
&= \frac{1}{1-z}\mathrm{Ehr}(P;z) \label{eq-ses-es} \\ 
&= \frac{h^*(P;z)}{(1-z)^{d+2}}. \label{eq-ses-h} 
\end{align} 
\end{thm}
\begin{thm}\label{thm-ses-nonneg}
The function $h^{\#}(P;z)$ is a polynomial in $z$ of the form 
\[ h^{\#}(P;z) = h^*(P;z),\] 
whose coefficients are nonnegative integers.
\end{thm}
Based on this theorem, we refer to $h^{\#}(P;z)$ as the {\itshape $h^{\#}$-polynomial} of $P.$
\begin{thm}\label{thm-ses-coeff}
Let 
\[ h^{\#}(P;z) = h^{\#}_{d+1}z^{d+1}+h^{\#}_{d}z^{d}+\cdots +h^{\#}_{1}z+h^{\#}_{0}.\] 
\begin{enumerate}
\item[\textup{(1)}]
We have $h^{\#}_{d+1} = 0$ and 
\[ h^{\#}_{i} = h^*_{i}\] 
for any $i \in \{ 0,1,\dots,d\}.$ 
Furthermore, 
\begin{align*} 
h^{\#}_0 &= 1, \\ 
\mathrm{vol}\,P &= \frac{1}{d!}\sum_{i = 0}^{d}h^{\#}_{i}, 
\end{align*} 
and 
\[ c = d+1-\deg h^{\#}(P;z).\] 
\item[\textup{(2)}]
The polytope $P$ is reflexive up to translation if and only if 
\[ h^{\#}_{i} = h^\#_{d-i}\] 
for any $i \in \{ 0,1,\dots,d\}.$
\end{enumerate}
\end{thm}
\begin{thm}\label{thm-ses-recip}
We have 
\begin{align} 
\mathrm{SE}(P^\circ;z) &= \frac{(-1)^{d+1}}{1-z}\mathrm{Ehr}\left( P;\frac{1}{z}\right) \label{eq-sesi-es} \\ 
&= \frac{(-1)^{d}}{z}\mathrm{SE}\left( P;\frac{1}{z}\right). \label{eq-sesi-ses} 
\end{align} 
\end{thm}
\begin{thm}\label{thm-se-inv}
Both $\mathcal H(P;t)$ and $\mathrm{SE}(P;z)$ are invariant under lattice equivalence.
\end{thm}
\begin{thm}\label{thm-se-pie}
Both $\mathcal H(P;t)$ and $\mathrm{SE}(P;z)$ satisfy the valuation property on the algebra of convex lattice polytopes. 
That is, they define linear functionals on the Grothendieck group of convex lattice polytopes.
\end{thm}
\section{Proofs of Main Results}
In this section, we prove Theorems \ref{thm-sep-recip}--\ref{thm-se-pie}. 
\begin{lem}\label{lem-hori-sub}
Let $t > 0$ be an integer. 
For any $i \in \{ 0,1,\dots,t-1\},$ the number of horizontal gigantures of $P$ in $tP^\circ$ with scale factor $i$ equals $\mathcal E(P^\circ;t-i).$
\end{lem}
\begin{proof}
Let $a \in \mathbb Z^n.$ 
By the convexity of $P,$ a giganture $iP+a$ is contained in $tP^\circ$ if and only if the translation vector $a$ lies in the relative interior of $(t-i)P.$ 
Therefore, the desired number equals 
\[\#\{ a \in \mathbb Z^n \mid iP+a \subset tP^\circ\} = \# ((t-i)P^\circ\cap\mathbb Z^n) = \mathcal E(P^\circ;t-i). \qedhere\] 
\end{proof}
\begin{proof}[Proof of Theorem \ref{thm-sep-recip}]
By Lemma \ref{lem-hori-sub}, the Ehrhart--Macdonald reciprocity law (see \cite[Theorem 4.1]{BR15}), and the identity \eqref{eq-sep-eppyr}, we obtain 
\begin{align*} 
\mathcal H(P^\circ;t) &= \sum_{i = 0}^{t-1}\mathcal E(P^\circ;t-i) \\ 
&= \sum_{i = 1}^{t}\mathcal E(P^\circ;i) \\ 
&= \mathcal E((\mathrm{Pyr}\,P)^\circ;t+1) \\ 
&= (-1)^{d+1}\mathcal E(\mathrm{Pyr}\,P;-t-1) \\ 
&= (-1)^{d+1}\mathcal H(P;-t-1). \qedhere 
\end{align*} 
\end{proof}
\begin{rmk}
Let $\mathcal H^+(P^\circ;t)$ denote the number of nondegenerate horizontal gigantures of $P$ in $tP^\circ$ for each integer $t > 0.$ 
Then, we obtain 
\[\mathcal H^+(P^\circ;t) = \mathcal E((\mathrm{Pyr}\,P)^\circ;t) = (-1)^{d+1}\mathcal H^+(P;1-t)\] 
by the same argument as above.
\end{rmk}
\begin{proof}[Proof of Theorem \ref{thm-ses-ser}]
By the identities \eqref{eq-sep-eppyr} and 
\begin{equation} 
\mathrm{Ehr}(\mathrm{Pyr}\,P;z) = \frac{1}{1-z}\mathrm{Ehr}(P;z) \label{eq-esp-es} 
\end{equation} 
(see \cite[Theorem 2.4]{BR15}), we obtain 
\begin{align*} 
\mathrm{SE}(P;z) &= \sum_{t = 0}^{\infty}\mathcal E(\mathrm{Pyr}\,P;t)z^t \\ 
&= \mathrm{Ehr}(\mathrm{Pyr}\,P;z) \\ 
&= \frac{1}{1-z}\mathrm{Ehr}(P;z) \\ 
&= \frac{h^*(P;z)}{(1-z)^{d+2}}. \qedhere 
\end{align*} 
\end{proof}
\begin{proof}[Proof of Theorem \ref{thm-ses-nonneg}]
Combining the identity \eqref{eq-ses-h} and Stanley's nonnegativity theorem for $h^*(P;z)$ (see \cite[Theorem 3.12]{BR15}), we obtain the desired result.
\end{proof}
\begin{proof}[Proof of Theorem \ref{thm-ses-coeff}]
\begin{enumerate}
\item[(1)]
Combining Theorem \ref{thm-ses-nonneg} and the known results for $h^*(P;z),$ we obtain the desired result.
\item[(2)]
Combining Theorem \ref{thm-ses-nonneg} and Hibi's palindromic theorem (see \cite[Theorem 4.6]{BR15}), we obtain the desired result. \qedhere
\end{enumerate}
\end{proof}
\begin{proof}[Proof of Theorem \ref{thm-ses-recip}]
By the identity \eqref{eq-sepi-epi}, the Ehrhart--Macdonald reciprocity law 
\[\mathrm{Ehr}(P^\circ;z) = (-1)^{d+1}\mathrm{Ehr}\left( P;\frac{1}{z}\right)\] 
(see \cite[Theorem 4.4]{BR15}), and the identities \eqref{eq-esp-es} and \eqref{eq-ses-es}, we have 
\begin{align*} 
\mathrm{SE}(P^\circ;z) &= \sum_{t = 1}^{\infty}\mathcal E((\mathrm{Pyr}\,P)^\circ;t+1)z^t \\ 
&= \frac{1}{z}\sum_{t = 2}^{\infty}\mathcal E((\mathrm{Pyr}\,P)^\circ;t)z^t \\ 
&= \frac{1}{z}\sum_{t = 1}^{\infty}\mathcal E((\mathrm{Pyr}\,P)^\circ;t)z^t \\ 
&= \frac{1}{z}\mathrm{Ehr}((\mathrm{Pyr}\,P)^\circ;z) \\ 
&= \frac{(-1)^{d+2}}{z}\mathrm{Ehr}\left(\mathrm{Pyr}\,P;\frac{1}{z}\right) \\ 
&= \frac{(-1)^{d}}{z}\cdot\frac{1}{1-z^{-1}}\mathrm{Ehr}\left( P;\frac{1}{z}\right) \\ 
&= \frac{(-1)^{d}}{z}\mathrm{SE}\left( P;\frac{1}{z}\right), 
\end{align*} 
which implies \eqref{eq-sesi-es}. 
Here, the third equality holds, since $(\mathrm{Pyr}\,P)^\circ$ contains no lattice points.
\end{proof}
\begin{rmk}
We define $\mathrm{SE}^+(P;z)$ and $\mathrm{SE}^+(P^\circ;z)$ by 
\[\mathrm{SE}^+(P;z) := \sum_{t = 0}^{\infty}\mathcal H^+(P;t)z^t \quad\text{and}\quad \mathrm{SE}^+(P^\circ;z) := \sum_{t = 1}^{\infty}\mathcal H^+(P^\circ;t)z^t.\] 
Then, we obtain 
\[\mathrm{SE}^+(P;z) = z\cdot\mathrm{Ehr}(\mathrm{Pyr}\,P;z) = \frac{z}{1-z}\mathrm{Ehr}(P;z) = \frac{z\cdot h^*(P;z)}{(1-z)^{d+2}}\] 
and 
\[\mathrm{SE}^+(P^\circ;z) = (-1)^{d}z\cdot\mathrm{SE}^+\left( P;\frac{1}{z}\right) = \frac{(-1)^{d+1}z}{1-z}\mathrm{Ehr}\left( P;\frac{1}{z}\right)\] 
by the same argument as above.
\end{rmk}
\begin{proof}[Proof of Theorem \ref{thm-se-inv}]
Combining \eqref{eq-sep-epsum}, \eqref{eq-ses-es}, and the fact that $\mathcal E(P;t)$ and $\mathrm{Ehr}(P;z)$ are invariant under lattice equivalence, we obtain the desired result.
\end{proof}
\begin{proof}[Proof of Theorem \ref{thm-se-pie}]
Let $P,$ $Q \subset \mathbb R^n$ be convex lattice polytopes such that $P\cup Q$ and $P\cap Q$ are also convex lattice polytopes. 
Combining \eqref{eq-sep-epsum} and the valuation property for $\mathcal E(P;t)$ and $\mathrm{Ehr}(P;z),$ we obtain 
\begin{align*} 
\mathcal H(P\cup Q;t) &= \mathcal H(P;t)+\mathcal H(Q;t)-\mathcal H(P\cap Q;t) 
\intertext{and} 
\mathrm{SE}(P\cup Q;z) &= \mathrm{SE}(P;z)+\mathrm{SE}(Q;z)-\mathrm{SE}(P\cap Q;z). \qedhere 
\end{align*} 
\end{proof}
\begin{rmk}
Similarly to Theorems \ref{thm-se-inv} and \ref{thm-se-pie}, both $\mathcal H^+(P;t)$ and $\mathrm{SE}^+(P;z)$ as well are invariant under lattice equivalence and satisfy the valuation property.
\end{rmk}
\section{Examples}\label{sec-eg}
In this section, we provide a few examples of summatory Ehrhart polynomials and series of convex lattice polytopes.
\begin{eg}
In $\mathbb R,$ let $P = [0,1].$ 
Then 
\[\mathcal H(P;t) = \sum_{i = 0}^{t}\mathcal E(P;i) = \sum_{i = 0}^{t}(i+1) = \sum_{i = 1}^{t+1}i = \frac{1}{2}(t+1)(t+2)\] 
and 
\[\mathcal H(P^\circ;t) = (-1)^2\mathcal H(P;-t-1) = \frac{1}{2}t(t-1).\] 
Furthermore, 
\[\mathrm{SE}(P;z) = \frac{1}{(1-z)^3}\] 
and 
\[\mathrm{SE}(P^\circ;z) = \frac{-1}{z}\cdot\frac{1}{(1-z^{-1})^3} = \frac{z^2}{(1-z)^3}\] 
by \eqref{eq-sesi-ses}.
\end{eg}
\begin{eg}
Let $a,$ $b > 0$ be integers. 
In $\mathbb R^2,$ let $P$ be the triangle with vertices $(0,0),$ $(a,0),$ and $(0,b),$ let $Q$ be the triangle with vertices $(a,b),$ $(a,0),$ and $(0,b),$ and let $R = [0,a]\times [0,b].$ 
Then 
\begin{align*} 
\mathcal H(R;t) &= \mathcal H(P\cup Q;t) \\ 
&= \mathcal H(P;t)+\mathcal H(Q;t)-\mathcal H(P\cap Q;t) \\ 
&= 2\mathcal H(P;t)-\mathcal H(P\cap Q;t) 
\end{align*} 
by Theorem \ref{thm-se-pie}. 
In addition, 
\begin{align*} 
\mathcal H(R;t) &= \sum_{i = 0}^{t}\mathcal E(R;i) \\ 
&= \sum_{i = 0}^{t}(ai+1)(bi+1) \\ 
&= ab\sum_{i = 0}^{t}i^2+(a+b)\sum_{i = 0}^{t}i+\sum_{i = 0}^{t}1 \\ 
&= ab\cdot\frac{1}{6}t(t+1)(2t+1)+(a+b)\cdot\frac{1}{2}t(t+1)+(t+1) \\ 
&= \frac{1}{6}(t+1)(2abt^2+(ab+3a+3b)t+6) 
\end{align*} 
and  
\begin{align*} 
\mathcal H(R^\circ;t) &= (-1)^3\mathcal H(R;-t-1) \\ 
&= \frac{1}{6}t(2abt^2+3(ab-a-b)t+ab-3a-3b+6) 
\end{align*} 
by \eqref{eq-sepi-sep}. 
If $a$ and $b$ are coprime, then 
\[\mathcal H(P\cap Q;t) = \frac{1}{2}(t+1)(t+2),\] 
since $P\cap Q$ is the line segment that contains only the lattice points $(0,0)$ and $(a,b)$ at its endpoints. 
Therefore, in this case, we have 
\begin{align*} 
\mathcal H(P;t) &= \frac{1}{2}(\mathcal H(R;t)+\mathcal H(P\cap Q;t)) \\ 
&= \frac{1}{2}\left(\frac{1}{6}(t+1)(2abt^2+(ab+3a+3b)t+6)+\frac{1}{2}(t+1)(t+2)\right) \\ 
&= \frac{1}{12}(t+1)(2abt^2+(ab+3a+3b+3)t+12) 
\end{align*} 
and 
\begin{align*} 
\mathcal H(P^\circ;t) &= (-1)^3\mathcal H(P;-t-1) \\ 
&= \frac{1}{12}t(2abt^2+3(ab-a-b-1)t+(a-3)(b-3)) 
\end{align*} 
by \eqref{eq-sepi-sep}. 
Furthermore, in the case where $a = b = 1,$ we have 
\begin{align*} 
\mathrm{SE}(P;z) &= \frac{1}{(1-z)^4}, \\ 
\mathrm{SE}(R;z) &= \frac{1+z}{(1-z)^4}, \\ 
\mathrm{SE}(P^\circ;z) &= \frac{(-1)^2}{z}\cdot\frac{1}{(1-z^{-1})^4} = \frac{z^3}{(1-z)^4}, 
\intertext{and}
\mathrm{SE}(R^\circ;z) &= \frac{(-1)^2}{z}\cdot\frac{1+z^{-1}}{(1-z^{-1})^4} = \frac{z^2(1+z)}{(1-z)^4} 
\end{align*} 
by \eqref{eq-sesi-ses}.
\end{eg}


\begin{thebibliography}{9}
\bibitem{BR15}
M.~Beck and S.~Robins, {\itshape Computing the Continuous Discretely: Integer Point Enumeration in Polyhedra}, 2nd ed., Undergraduate Texts in Mathematics, Springer, New York, 2015. 
\bibitem{Hir22}
T.~Hirotsu, Normal-sized hypercuboids in a given hypercube, preprint, arXiv:2211.15342.
\bibitem{Hir25}
T.~Hirotsu, Average-sized miniatures and normal-sized miniatures of lattice polytopes, preprint, arXiv:2501.00459.
\bibitem{Hir26}
T.~Hirotsu, Horizontal miniatures and normal-sized miniatures of convex lattice polytopes, preprint, arXiv:2605.20905.
\end{thebibliography}
\end{document}